\documentclass[11pt]{article}
\usepackage{amsfonts, amssymb, latexsym, amsmath, graphicx}
\newenvironment{proof}           {\noindent{\bf Proof.} }%
                                 {\null\hfill$\Box$\par\medskip}

\newtheorem{theorem}{Theorem}[section]

\newtheorem{problem}[theorem]{Problem}
\newtheorem{conjecture}[theorem]{Conjecture}
\newtheorem{lemma}[theorem]{Lemma}

\makeatletter
\def\imod#1{\allowbreak\mkern10mu({\operator@font mod}\,\,#1)}
\makeatother

\def\v3{\vskip 3 ex}

\begin{document}
\title{Extremal Restraints for Graph Colourings}
\author{Jason I. Brown\footnote{Communicating author.}~, Aysel Erey and Jian Li \\
Department of Mathematics and Statistics\\
Dalhousie University\\
Halifax, Nova Scotia, Canada B3H 3J5}
\date{}

\maketitle

\begin{abstract}
A {\em restraint} on a (finite undirected) graph $G = (V,E)$ is a function $r$ on $V$ such that $r(v)$ is a finite subset of ${\mathbb N}$; a proper vertex colouring $c$ of $G$ is {\em permitted} by $r$ if $c(v)  \not\in r(v)$ for all vertices $v$ of $G$ (we think of $r(v)$ as the set of colours {\em forbidden} at $v$). Given a large number of colors, for restraints $r$ with exactly one colour forbidden at each vertex the smallest number of colorings is permitted when $r$ is a constant function, but the  problem of what restraints permit the largest number of colourings is more difficult. We determine such extremal restraints for complete graphs and trees.
\end{abstract}

\v3

\v3
\section{Introduction}

A (proper vertex) {\em k-colouring} of a finite undirected graph $G$ is a function $f:=V(G) \rightarrow \{1,2,\ldots,k\}$ such that for every edge $e = uv$ of $G$, $f(u) \neq f(v)$ (we will denote by $[k] = \{1,2,\ldots,k\}$ the set of {\em colours}).
There are variants of vertex colourings that have been of interest. In a {\em list colouring}, for each vertex $v$ there is a finite list (that is, set) $L(v)$ of colours available for use, and then one wishes to properly colour the vertices such that the colour of $v$ is from $L(v)$. If $|L(v)|=k$ for every vertex $v$, then a list colouring is called a {\em k-list colouring}.  There is a vast literature on list colourings (see, for example, \cite{alon} and \cite{chartrand}, Section 9.2).

We are going to consider a complementary problem, namely colouring the vertices of a graph $G$ where each vertex $v$ has a {\em forbidden} finite set of colours, $r(v) \subset {\mathbb N}$ (we allow $r(v)$  to be equal to the empty set); we call the function $r$ a {\em restraint} on the graph. A restraint $r$ is called an {\em $m$-restraint} if $|r(u)| \leq m$ for every $u\in V(G)$, and $r$ is called a {\em standard $m$-restraint} if $|r(u)| = m$ for every $u\in V(G)$. If $m = 1$ (that is, we forbid at most one colour at each vertex) we omit $m$ from the notation and use the word {\em simple} when discussing such restraints. 

A $k$-colouring $c$ of $G$ is {\em permitted} by restraint $r$ (or $c$ is a colouring {\em with respect to $r$}) if for all vertices of $v$ of $G$, $c(v) \not\in r(v)$.
Restrained colourings arise in a natural way as a graph is sequentially coloured, since the colours already assigned to vertices induce a set of forbidden colours on their uncoloured neighbours. Restrained colourings can also arise in scheduling problems where certain time slots are unavailable for certain nodes (c.f.\ \cite{kubale}). Moreover, restraints are of use in the construction of critical graphs (with respect to colourings) \cite{toft}; a $k$-chromatic graph $G = (V,E)$ is said to be {\em $k$-amenable} iff for every non-constant simple restraint $r:V \rightarrow \{1,2,\ldots,k\}$ permits a $k$-colouring \cite{amenable,roberts}.  Finally, observe that if each vertex $v$ of a graph $G$ has a list of available colours $L(v)$, and, without loss,
\[ L = \bigcup_{v \in V(G)} L(v) \subseteq [k] \]
then setting $r(v) = \{1,2,\ldots,k\} - L(v)$ we see that $G$ is list colourable with respect to the lists $L(v)$ iff $G$ has a $k$-colouring permitted by $r$.

The well known {\em chromatic polynomial} $\pi(G,x)$ (see, for example, \cite{chrompolybook}) counts the number of $x$-colourings of $G$ with $x$ colours. Given a restraint $r$ on graph $G$, we define  the {\em restrained chromatic polynomial} of $G$ with respect to $r$, $\pi_{r}(G,x)$, to be the number of $x$-colourings permitted by restraint $r$. Note that this function extends the definition of chromatic polynomial, as if $r(v) = \emptyset$ for all vertices $v$, then  $\pi_r(G,x) = \pi(G,x)$.

Using standard techniques (i.e. the deletion/contraction formula), we can show that the restrained chromatic polynomial $\pi_{r}(G,x)$ is a polynomial function of $x$ for $x$ sufficiently large, and like chromatic polynomials, the restrained chromatic polynomials of a graph $G$ of order $n$ is monic of degree $n$ with integer coefficients that alternate in sign, but unlike chromatic polynomials, the constant term need not be 0 (we can show that the constant term for any restraint $r$ on $\overline{K_{n}}$ is $(-1)^{n}\prod_{v \in V(G)} |r(v)|)$. Also, note that if $r$ is a constant standard $m$-restraint, say $r(v) = S$ for all $v \in V$, then $\pi_{r}(G,x) = \pi(G,x-m)$ for $x$ at least as large as $\mbox{max}(S)$.

Observe that if $r'$ arises from $r$ by a permutation of colours, then $\pi_r(G,x)=\pi_{r'}(G,x)$ for all $x$ sufficiently large. Thus if $\displaystyle{k = \sum_{v \in V(G)} |r(v)|}$ then we can assume (as we shall do for the rest of this paper) that each $r(v) \subseteq [k]$, and so there are only finitely many restrained chromatic polynomials on a given graph $G$. Hence past some point (past the roots of all of the differences of such polynomials), one polynomial exceeds (or is less) than all of the rest, no matter what $x$ is.

As an example, consider the cycle $C_{3}$. There are essentially three different kinds of standard simple restraints on $C_{3}$, namely $r_{1}= [\{1\}, \{1\}, \{1\}]$,  $r_{2} = [\{1\}, \{2\}, \{1\}]$ and $r_{3}= [\{1\}, \{2\}, \{3\}]$ (If the vertices of $G$ are ordered as $v_1,v_2\dots v_n$, then we usually write $r$ in the form $[r(v_1),r(v_2)\dots r(v_n)]$). For $x\geq 3$, the restrained chromatic polynomials with respect to these restraints can be calculated as
\begin{eqnarray*}
\pi_{r_1}(C_3,x) & = & (x-1)(x-2)(x-3),\\
\pi_{r_2}(C_3,x) & = & (x-2)(x^2-4x+5), \mbox{ and} \\
\pi_{r_3}(C_3,x) & = & 2(x-2)^2+(x-2)(x-3)+(x-3)^3.
\end{eqnarray*}
where $\pi_{r_1}(C_3,x)<\pi_{r_2}(C_3,x)<\pi_{r_3}(C_3,x)$ holds  for $x>3.$

Our focus in this paper is on the following: given a graph $G$ and $x$ large enough, what standard  simple restraints permit the largest/smallest number of $x$-colourings?
In the next section, we  give a complete answer to minimization part of this question, and then turn our attention to the more difficult maximization problem, and in the case of complete graphs and  trees, describe the standard simple restraints which permit the largest number of colourings.

\section{Standard Restraints permitting the extremal number of colourings}

The standard $m$-restraints that permit the smallest number of colourings are easy to describe, and  are, in fact, the same for all graphs. In \cite{carsten} (see also \cite{donner}) it was proved that if a graph $G$ of order $n$ has a list of at least $k$ available colours at every vertex, then the number of list colourings is at least $\pi(G,k)$ for any natural number $k\geq n^{10}$.
As we already pointed out, given a standard $m$-restraint $r$ on a graph $G$ and a natural number $x\geq mn$, we can consider an $x$-colouring permitted by $r$ as a list colouring $L$ where each vertex $v$ has a list $L(v)=[x]-r(v)$ of $x-m$ available colours. Therefore, we derive that for a standard  $m$-restraint $r$ on graph $G$, $\pi_r(G,x) \geq \pi(G,x-m)$ for any natural number $x\geq n^{10}+mn$. But $\pi_{r_{const}^m}(G,x)$ is clearly the number of colourings permitted by the {\em constant} standard $m$-restraint in which $\{1,2,\ldots,m\}$ is restrained at each vertex. In particular, for any graph $G$, the constant standard simple restraints always permit the smallest number of colourings (provided the number of colours is large enough).

The more difficult question is which standard $m$-restraints permit the largest number of colorings; even for standard simple restraints, it appears difficult, so we will focus on this question. As we shall see, the extremal simple restraints differ from graph to graph. We investigate the extremal problem for two important families of graphs: complete graphs and trees.

\subsection{Complete graphs}

First, we prove that for complete graphs, the standard simple restraints that allow for the largest number of colourings are obtained when all vertices have different restrained colours.

\begin{theorem}\label{completethm}
Let $r:  \{ v_{1},  v_{2}, \ldots, v_{n} \} \longrightarrow [n]$ be any standard simple restraint on $K_{n}$ , then for all $x \geq n$, $ \pi_{r}(K_{n}, x) \le \pi_{r'}(K_{n}, x)$, where  $r'(v_{i})=i$  for all $i \le n$.
\end{theorem}
\begin{proof}
We show that if two vertices of a complete graph have the same forbidden colour, then we can improve the situation, colouring-wise, by reassigning the restraint at one of these vertices to a colour not forbidden elsewhere.
Let $r_{1}:  \{ v_{1},  v_{2}, \ldots, v_{n} \} \longrightarrow [n]$ be a standard simple restraint on $K_{n}$ with $r_{1}(v_{i}) = r_{1}(v_{j})= t$, and there is an element $l \in [n]$  such that $l \notin r(V(K_{n}))$. Then  setting
\[ r'_{1}(v_{s}) = \left\{ \begin{array}{ll}
                       r(v_{s}) & \mbox{ if } s \neq j\\
                       l        & \mbox{ if } s = j
                       \end{array} \right.
\]
we will show that $\pi_{r_{1}}(K_{n}, x) \le \pi_{r'_{1}}(K_{n}, x)$ for $x \ge n$.

Let $c$ be a proper $x$-colouring of $K_{n}$ permitted by $r_{1}$. We produce for each such $c$ another proper $x$-colouring $c'$ of $K_{n}$ permitted by $r'_{1}$, in a 1--1 fashion.
We take cases based on $c$.

\begin{itemize}
  \item case 1:   $c(v_{j}) \neq l$. The proper $x$-colouring $c$ is also permitted by $r'_{1}$, so take $c'=c$.
  \item case 2:   $c(v_{j}) = l$ and  $t$ is not used by $c$ on the rest of $K_{n}$.  Let $c'$ be the proper $x$-colouring of $K_{n}$ with $c'(v_{u})= c(v_{u})$ if $u \neq j$ and $c'(v_{j})= t$. This gives us a proper $x$-colouring $c'$ permitted by $r'_{1}$.
 \item case 3:   $c(v_{j}) = l$ and  $t$ is used somewhere on the rest of $K_{n}$ by $c$, say vertex $v_{k}$. Let $c'$ be a proper $x$-colouring of $K_{n}$ with $c'(v_{u}) = c(v_{u})$ if $ u \neq  j$ or $k$, $c'(v_{j})= t$ and $c'(v_{k})=l$.  This gives us a proper $x$-colouring $c'$ permitted by $r'_{1}$.
 \end{itemize}

No colouring from one case is a colouring in another case and  different colourings $c$ give rise to different colourings $c'$ within each case. Therefore, we have $\pi_{r_{1}}(K_{n}, x) \le \pi_{r'_{1}}(K_{n}, x)$ for $x \ge n$.

If $r$ is not 1-1, we start with $r_{1} = r$ and repeat the argument until we arrive at a simple restraint $r^{\ast}$ that is 1-1 on $V(G)$ and $\pi_{r}(K_{n}, x) \le \pi_{r^{\ast}}(K_{n}, x)$ for $x \ge n$. Clearly $r^{\ast}$ arises from $r'$ by a permutation of colours, so $\pi_{r}(K_{n}, x) \le \pi_{r^{\ast}}(K_{n}, x) = \pi_{r'}(K_{n}, x)$ for $x \ge n$ and we are done.
\end{proof}

\subsection{Trees}

We now consider extremal simple restraints for trees, but first we need some notation.
Suppose $G$ is a connected bipartite graph with bipartition $(A,B)$. Then a standard simple  restraint is called an \textit{alternating restraint}, denoted $r_{alt}$, if $r_{alt}$ is constant on both $A$ and $B$ individually but $r_{alt}(A)\neq r_{alt}(B)$. We show that for trees alternating restraints permit the largest number of colorings.

Before we begin, though, we will need some notation and a lemma. If $r$ is a restraint on $G$ and $H$ is an induced subgraph of $G$ then $r|_H$, the {\em restriction of $r$ to $H$}, denotes the restraint function induced by $r$ on the vertex set of $H$ (if $A$ is a vertex subset of $G$ then $G_A$ is the subgraph induced by $A$).

\begin{lemma}\label{tree3}
Let $T$ be a tree on $n$ vertices and $r:V(T)\rightarrow [n]$ be a $2$-restraint such that there is at most one vertex $w$ of $T$ with $|r(w)| = 2$. Then for any $k \geq \operatorname{max}\{3,n\}$, $\pi_r(T,k) > 0$.
\end{lemma}
\begin{proof} The proof is by induction on $n$. For $n= 1$ the proof is trivial, so we assume that $n \geq 2$. As $T$ has at least two leaves, let $u$ be a leaf of $T$ such that $|r(u)|\leq 1$ and $v$ be the stem of $u$. By induction we can colour $T-u$ with respect to $r|_{T-u}$. As $k \geq 3$, there is a colour different from $r(u)$ and the colour assigned to $v$, so we can extend the colouring to one permitted by $r$ on all of $T$.
\end{proof}

\begin{theorem}\label{treemaxmin}Let $T$ be a tree on $n$ vertices and $r:V(T)\rightarrow [n]$ be a standard simple restraint that is not an alternating restraint, then for $k \geq n$, $$\pi_r(T,k) < \pi_{r_{alt}}(T,k).$$
\end{theorem}
\begin{proof}
We proceed by induction on $n$. We leave it to the reader to check the basis step $n=2$. Suppose that $n\geq 3$, $u$ be a leaf of $T$ and $v$ be the neighbor of $u$. Also, let $v_1,v_2,\dots v_m$ be the vertices of the set $N(v)-\{u\}$. Let $T'=T-u$ and $T''=T-\{u,v\}$. Let $T^i$ be the connected component of $T''$ which contains the vertex $v_i$. Given a simple restraint $r$ on $T$, we consider two cases:
\begin{itemize}
\item case 1: $r(u)=r(v).$\newline
Once all the vertices of $T'$ are coloured with respect to $r|_{T'}$, $u$ has $k-2$ choices because it cannot get the colour $r(u)$ and the colour assigned to $v$ which different from $r(u)$. Thus,

\begin{equation}\label{treecase1}
\pi_r(T,k)=(k-2)\pi_{r|_{T'}}(T',k).
\end{equation}

\item case 2: $r(u)\neq r(v).$\newline
In this case we define $x_{n-1}^r$ (respectively $y_{n-1}^r$) to be the number of $k$-colourings of $T'$ permitted by $r|_{T'}$ where $v$ gets (respectively does not get) the colour $r(u)$. Now it can be verified that $\pi_{r|_{T'}}(T',k)=x_{n-1}^r+y_{n-1}^r$ and $\pi_r(T,k)=(k-1)x_{n-1}^r+(k-2)y_{n-1}^r$. In other words,
\begin{equation}\label{treecase2}
\pi_r(T,k)=(k-2)\pi_{r|_{T'}}(T',k)+x_{n-1}^r
\end{equation}
Also let us define a restraint function $r_i:V(T^i)\rightarrow {\mathbb N} $ on each component $T^i$ for $i=1,\dots m$ as follows:
\begin{enumerate}
\item If $r(v_i)=r(u)$ then $r_i(w):=r(w)$ for each $w\in V(T^i)$
\item If $r(v_i)\neq r(u)$ then  \[r_i(w) := \left\{ \begin{array}{ll}
                      \{r(v_i),r(u)\} & \mbox{ if  $w=v_i$}\\
                       r(w)        & \mbox{ if $w\neq v_i$ }
                       \end{array} \right. \]
 \textit{for each } $w\in V(T^i).$
\end{enumerate}
Now,  $\displaystyle{x_{n-1}^r=\prod_{i=1}^m \pi_{r_i}(T^i,k)}$ which is strictly larger than $0$ by Lemma~\ref{tree3}.
\end{itemize}

By comparing Equations (\ref{treecase1}) and (\ref{treecase2}), it is clear that  $\pi_r(T,k)$ will be maximized in case 2. Since $r(V(T^i))\subseteq r_i(V(T^i))$,  $\pi_{r_i}(T,k)$ is maximized when $r(v_i)=r(u)$, that is, when $r_i$ and $r|_{T^i}$ are equal to each other for each $i=1\dots m$. Moreover,  $\pi_{r|_{T^i}}(T^i,k)$ is maximized when $r|_{T^i}$ is alternating on $T^i$ for each $i=1\dots m$, and $\pi_{r|_{T'}}(T',k)$ is maximized when $r|_{T'}$ is alternating on $T'$ by the induction hypothesis. Hence, $\pi_r(T,k)$ attains its maximum value when $r$ is alternating on $T$. Moreover, this value is strictly larger than all the others. Therefore, the result follows.
\end{proof}

\section{Concluding remarks and open problems}

It is worth noting that for complete graphs and trees the simple restraints which maximize the restrained chromatic polynomials are all minimal colourings, that is, colourings with the smallest number of colours. One might wonder therefore  whether this always holds, but unfortunately this is not always the case. For consider the graph $G$ in Figure~\ref{twotriangles} which has chromatic number $3$. It is easy to see that there is essentially only one standard simple restraint ($r_2=[1,2,3,1,2,3]$) which is a proper colouring of the graph with three colours.  If $r_1=[1,2,3,1,2,4]$, then some direct computations show that  $$\pi_{r_1}(G,x)-\pi_{r_2}(G,x)=(x-3)^2>0$$ for all $x$ large enough. It follows that  the simple restraint which maximizes the restrained chromatic polynomial of $G$ cannot be a minimal colouring of the graph.

 \begin{figure} [ht]
  \begin{center}
   \includegraphics[width=4in]{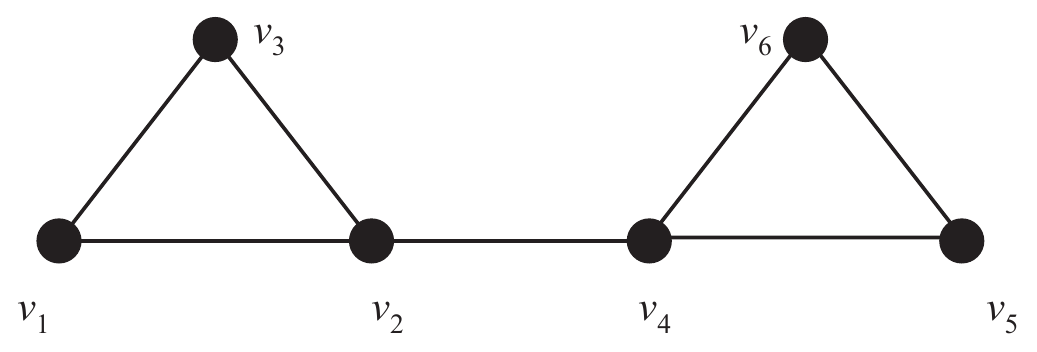}
   \caption{Graph whose standard simple restraint permitting the largest number of colourings is not a minimal colouring.}
  \label{twotriangles}
  \end{center}
  \end{figure}

We believe, however, that for bipartite graphs the simple restraint which maximizes the restrained chromatic polynomial is a minimal colouring of the graph. More specifically, we propose the following:

\begin{conjecture}\label{bipartiteconjecture}
Let $r:  \{ v_{1},  v_{2}, \ldots, v_{n} \} \longrightarrow [n]$ be any standard simple restraint  on a bipartite graph $G$ and $x$ large enough. Then,
 $\pi_{r}(G, x) \le  \pi_{r_{alt}}(G, x)$.
\end{conjecture}

We verified that the conjecture above is correct for all such graphs of order at most $6$. Indeed, we know that among all graphs of order at most $6$, there are only two graphs where the standard simple restraint which maximizes the restrained chromatic polynomial is not a minimal colouring of the graph. Therefore, we suggest the following interesting problem:

\begin{problem} Is it true that for almost all graphs the standard simple restraint which maximizes the restrained chromatic polynomial is a minimal colouring of the graph?
\end{problem}
\vskip0.4in

\noindent {\bf \large Acknowledgments} \\
The authors would like to thank the referee for his help and insightful comments.
This research was partially supported by a grant from NSERC.

\end{document}